\documentclass{amsart}
\usepackage{amssymb,amsmath,amscd, amsfonts,latexsym}
\usepackage{graphicx}
\usepackage{url}
\usepackage[top = 3.9cm, bottom = 3.9cm, left = 4cm, right = 4cm]{geometry}
\usepackage{color}
\newtheorem{theorem}{Theorem}

\newtheorem{lemma}{Lemma}

\theoremstyle{remark}

\newcommand{\bbQ}{\mathbb Q}

\title[\title{On a Problem of Pillai with {$k$}-Fibonacci Numbers and Powers of $2$}]{On a Problem of Pillai with $k$-Generalized Fibonacci Numbers and Powers of $2$}

\author[Mahadi Ddamulira, Carlos A. G\'omez and Florian Luca]{Mahadi Ddamulira, Carlos A. G\'omez and Florian Luca}

\subjclass[2010]{11D61,11B39,11D45, 11J86}

\keywords{Diophantine equations, Pillai's problem, Generalized Fibonacci sequence, Reduction method}

\address{Mahadi Ddamulira \newline
         \indent Institute of Analysis and Number Theory, Graz University of Technology \newline
        \indent Steyrergasse 30/II \newline
         \indent A-8010 Graz, Austria}

\email{mddamulira\char'100tugraz.at}

\address{Carlos A. G\'omez \newline
         \indent Departamento de Matem\'aticas, Universidad del Valle \newline
        \indent Calle 13 No 100--00 \newline
         \indent Cali, Colombia}

\email{carlos.a.gomez\char'100correounivalle.edu.co}

\address{Florian Luca \newline
         \indent School of Mathematics, University of the Witwatersrand \newline
        \indent Private Bag X3\newline
        \indent WITS 2050\newline
         \indent Johannesberg, South Africa}
         
\address{Max Planck Institute for Mathematics,\newline 
         \indent Vivatsgasse 7, 53111 Bonn, Germany}
         
\address{Department of Mathematics, Faculty of Sciences, University of Ostrava,\newline
           \indent 30 dubna 22, 701 03 Ostrava 1, Czech Republic}

\email{Florian.Luca\char'100wits.ac.za}
\begin{document}


\begin{abstract}
For an integer $ k\geq 2 $, let $ \{F^{(k)}_{n} \}_{n\geq 0}$ be the $ k$--generalized Fibonacci sequence which starts with $ 0, \ldots, 0, 1 $ ($ k $ terms) and each term afterwards is the sum of the $ k $ preceding terms. In this paper, we find all integers $c$ having at least two representations as a difference between a $k$--generalized Fibonacci number and a powers of 2 for any fixed $k \geqslant 4$. This paper extends previous work from \cite{Luca15}  for the case $k=2$ and \cite{Luca16} for the case $k=3$.
\end{abstract}

\maketitle

\section{Introduction}
A perfect power is a positive integer of the form $ a^{x} $ where $ a>1 $ and $ x\geq 2 $ are integers. Pillai wrote several papers on these numbers. In $ 1936 $ and again in $ 1945 $ (see \cite{Pillai:1936}, \cite{Pillai:1937}), he conjectured that for any given integer $ c\geq 1 $, the number of positive integer solutions $ (a,b,x,y) $, with $ x\geq 2 $ and $ y\geq 2 $, to the Diophantine equation
\begin{equation}\label{111}
a^{x}-b^{y} = c,
\end{equation}
is finite. This conjecture, which is still open for all $ c\neq 1 $, amounts to saying that the distance between two consecutive terms in the sequence of all perfect powers tends to infinity. The case $ c=1 $ is Catalan's conjecture which states that the only solution in positive integers to \eqref{111} for $a, b > 0$, $x, y > 1$  is $x = 2,~~ a = 3, ~~y = 3, ~~b = 2$. This conjecture was proved by  Mih{\u a}ilescu \cite{Mihailescu}.

Pillai's problem was continued in $ 1936 $ by Herschfeld (see \cite{Herschfeld:1935}, \cite{Herschfeld:1936}) who showed that if $ c $ is an integer with sufficiently large absolute value, then the equation \eqref{111}, in the special case $ (a,b) = (3,2) $, has at most one solution $ (x,y) $. For small $ |c| $ this is not the case. Pillai (see \cite{Pillai:1936}, \cite{Pillai:1937}) extended Herschfeld's result to the more general exponential Diophantine equation \eqref{111} with fixed integers $ a,b,c $ with gcd$(a,b)=1$ and $ a>b\geq 1 $. Specifically, Pillai showed that there exists a positive integer $ c_{0}(a,b) $ such that, for $ |c|> c_{0}(a,b) $, equation \eqref{111} has at most one integer solution $ (x,y) $.

Recently, Ddamulira, Luca and Rakotomalala \cite{Luca15} considered the Diophantine equation
\begin{equation}\label{eq:Luca1} 
F_n - 2^m = c,
\end{equation}
where $c$ is a fixed integer and $\{F_n\}_{n \geqslant 0}$ is the sequence of Fibonacci numbers given
by $F_0 = 0$, $F_1 = 1$ and $F_{n+2} = F_{n+1} + F_n$ for all $n \geqslant 0$.
This type of equation can be seen as a variation of Pillai's equation. Ddamulira et.al. proved
that the only integers $c$ having at least two representations of the form
$F_n - 2^m$ are contained in the set  $\mathcal C = \{ 0, -1, 1, -3, 5, -11, -30, 85 \}$. Moreover, they
computed for each $c\in\mathcal C$ all representations of the from \eqref{eq:Luca1}.

Bravo, Luca and Yaz\'an \cite{Luca16} considered the Diophantine equation
\begin{equation}\label{eq:Luca2} T_n - 2^m = c, \end{equation}
where $c$ is a fixed integer and $\{T_n\}_{n \geqslant 0}$ is the sequence of Tribonacci numbers given
by $T_0 = 0$, $T_1 = 1$, $ T_{2} = 1 $ and $T_{n+3} = T_{n+2} + T_{n+1}+T_{n}$ for all $n \geqslant 0$.
In their paper, Bravo et. al. proved that the only integers $c$ having at least two representations of the form
$T_n - 2^m$ are contained in the set  $\mathcal C = \{ 0, -1, -3, 5, -8 \}$. In fact, each $c\in\mathcal C$ has exactly two representations of the from \eqref{eq:Luca2}.

In the same spirit, Chim, Pink and Ziegler \cite{cpz16} considered the Diophantine equation
\begin{equation}\label{eq:Salz} F_n - T_m = c, \end{equation}
where $c$ is a fixed integer. They proved that the only integers $c$ having at least two representations of the form
$F_n - T_m$ are contained in the set  $$\mathcal C = \{ 0, 1, -1, -2, -3, 4, -5, 6, 8, -10, 11, -11, -22, -23, -41, -60, -271\} .$$ In particular, they
computed for each $c\in\mathcal C$ all representations of the from \eqref{eq:Salz}, showing that each $c\in \mathcal C$ has at most four representations.

The purpose of this paper is to generalize the previous results corresponding to \eqref{eq:Luca1} and \eqref{eq:Luca2}. Let $k \geqslant 2$ be an integer. We consider a generalization of Fibonacci sequence called the $k$--generalized Fibonacci sequence $\lbrace F_n^{(k)} \rbrace_{n\geqslant 2-k}$ defined as
\begin{eqnarray}
F_n^{(k)} &=& F_{n-1}^{(k)} + F_{n-2}^{(k)} + \cdots + F_{n-k}^{(k)},\label{fibb1}
\end{eqnarray}
with the initial conditions
\[F_{-(k-2)}^{(k)} = F_{-(k-3)}^{(k)} = \cdots = F_{0}^{(k)} = 0 ~~ \text {  and  }~~ F_{1}^{(k)} = 1.\]
We call $F_{n}^{(k)}$ the $n$th $k$--generalized Fibonacci number. Note that when $k=2$, it is the classical Fibonacci number ($n$th term, which is denoted by $F_{n}$ here for simplicity) and when $k=3$ it is the Tribonacci number.

The first direct observation is that the first $ k+1 $ nonzero terms in $F_{n}^{(k)}  $ are powers of $ 2 $, namely
\begin{eqnarray*}
F_{1}^{(k)}=1,~ F_{2}^{(k)}=1,~ F_{3}^{(k)}=2, ~F_{4}^{(k)}=4,\ldots, F_{k+1}^{(k)}=2^{k-1},
\end{eqnarray*}
while the next term in the above sequence is $ F_{k+2}^{(k)}=2^{k}-1 $.
Thus, we have that
\begin{eqnarray}\label{Fibbo111}
F_{n}^{(k)} = 2^{n-2} ~~~~~~\textrm{ holds~ for~ all  }  ~~~ 2\leq n\leq k+1.
\end{eqnarray}
We also observe that the recursion \eqref{fibb1} implies the three--term recursion
\begin{eqnarray}
F_{n}^{(k)}&=&2F_{n-1}^{(k)} - F_{n-k-1}^{(k)} ~~~~\text{ for all  } ~ n\geq 3,\label{fibb2}
\end{eqnarray}
which shows that the $ k $--Fibonacci sequence grows at a rate less than $ 2^{n-2} $. In fact, the inequality $ F_{n}^{(k)}<2^{n-2} $ holds for all $ n\geq k+2 $ (see \cite{BBL17}, Lemma $ 2 $).

In this paper, we find all integers $c$ admitting at least two representations of the form $ F_n^{(k)} - 2^m $ for some positive integers $k$, $n$ and $m$. This can be interpreted as solving the equation
\begin{align}
\label{Problem}
 F_n^{(k)} - 2^m = F_{n_1}^{(k)} - 2^{m_1}~~~~(=c)
\end{align}
with $(n, m) \neq (n_1, m_1)$. As we already mentioned, the cases $k=2$ and $ k=3 $ have been solved completely by Ddamulira, Luca and Rakotomalala \cite{Luca15} and Bravo, Luca and Yaz\'an \cite{Luca16}, respectively. So, we focus on the case  $k \geqslant 4$.\\

We prove the following theorem:

\begin{theorem}\label{Main}
Assume that $k\geq 4$. Then equation \eqref{Problem} with $ n>n_{1}\geq 2 $, $ m>m_{1}\geq 0 $ has the following families of solutions $(c,n,m,n_1,m_1)$.
\begin{itemize}
\item[(i)] In the range $2 \le n_1 < n \le k+1$, we have the following solution:
$$
(0,s,s-2,t,t-2)\qquad {\text{ for }}\quad 2\le t<s\le k+1.
$$
\item[(ii)] In the ranges $2 \le n_1 \le k+1$ and $k+2 \le n \le 2k+2$,  we have the following solutions:\vspace*{.3cm}
\begin{enumerate}
\item[$(a)$] when $n_1 = n-1$:
$$
\left(2^{k-1}-1, k+2, k-1, k+1, 0\right)
$$
\item[$(b)$] when $n_1 < n-1$: 
$$
\qquad\left(2^{\gamma}-2^{\rho}, k+2^a-2^b,k+2^a-2^b-2,\gamma+2,\rho\right),
$$
with $\gamma = b-3+2^a-2^b$ and $\rho = a-3+2^a-2^b$, where $a>b\ge 0$, $(a,b)\ne (1,0)$ and $\gamma+3\le k+2$.

\end{enumerate}
\item[(iii)] In the range $k+2\le n_1<n\le 2k+2$, we have the following solutions: if the integer $a$ is maximal such that  $2^a\le k+2$ satisfies $ a+2^{a}=k+1+2^{b} $ for some positive integer $ b $, then
$$
(-2^{a+2^a -3},k+2^a,k+2^a-2,k+2^{b},b+2^{b}-3).
$$
\item[(iv)] If $n = 2k+3$, and additionally $k=2^t-3$ for some integer $t\ge3$, then:
$$
( 1 - 2^{t+2^t-3}, 2^{t+1}-3, 2^{t+1}-5, 2, t+2^t-3 ).
$$

\end{itemize}
Equation \eqref{Problem} has no solutions with $n>2k+3$.
\end{theorem}

\section{Preliminary Results}

Here, we recall some of the facts and properties of the $ k- $generalized Fibonacci sequence which will be used later in this paper.
It is known that the characteristic polynomial of the $k$--generalized Fibonacci numbers $F^{(k)}:=\lbrace F_n^{(k)} \rbrace_{n\geq 0}$, namely
\[ \Psi_k(x) := x^k - x^{k-1} - \cdots - x - 1, \]
is irreducible over $\bbQ[x]$ and has just one root outside the unit circle. Let $\alpha := \alpha(k)$ denote that single root, which is located between $2\left(1-2^{-k} \right)$ and $2$ (see \cite{Dresden2014}). To simplify notation, in our application we shall omit the dependence on $k$ of $\alpha$. We shall use $\alpha^{(1)}, \dotso, \alpha^{(k)}$ for all roots of $\Psi_k(x)$ with the convention that $\alpha^{(1)} := \alpha$.

We now consider for an integer $ k\geq 2 $, the function
\begin{eqnarray}\label{fun12}
f_{k}(z) = \dfrac{z-1}{2+(k+1)(z-2)} \qquad \text{for}~~~ z \in \mathbb{C}.
\end{eqnarray}
With this notation, Dresden and Du presented in  \cite{Dresden2014} the following ``Binet--like" formula for the terms of $F^{(k)}$:
\begin{eqnarray} \label{Binet}
F_n^{(k)} = \sum_{i=1}^{k} f_{k}(\alpha^{(i)}) {\alpha^{(i)}}^{n-1}.
\end{eqnarray}
It was proved in \cite{Dresden2014} that the contribution of the roots which are inside the unit circle to the formula (\ref{Binet}) is very small, namely that the approximation
\begin{equation} \label{approxgap}
\left| F_n^{(k)} - f_{k}(\alpha)\alpha^{n-1} \right| < \dfrac{1 }{2} \quad \mbox{holds~ for~ all~ } n \geqslant 2 - k.
\end{equation}
When $ k=2 $, one can easily prove by induction that
\begin{eqnarray}\label{Fib11}
\alpha^{n-2} \leq F_{n} \leq \alpha^{n-1} ~~\text{ ~for ~all ~ } n\geq 1.
\end{eqnarray}
It was proved by Bravo and Luca in \cite{BBL17} that
\begin{eqnarray}\label{Fib12}
\alpha^{n-2} \leq F_{n}^{(k)} \leq \alpha^{n-1} \text{  holds for all  } n\geq 1 \text{ and } k\geq 2,
\end{eqnarray}
which shows that \eqref{Fib11} holds for the $ k $--generalized Fibonacci numbers as well. The observations made from the expressions \eqref{Binet} to \eqref{Fib12} enable us to call $ \alpha $ the \textit{dominant root } of $ F^{(k)} $.

In order to prove our main result Theorem \ref{Main}, we need to use several times a Baker type lower bound for a nonzero linear form in logarithms of algebraic numbers and such a bound, which plays an important role in this paper, was given by Matveev \cite{MatveevII}. There are other explicit lower bounds for linear forms in logarithms of algebraic numbers in the literature, like that by Baker and W{\"u}stholz in  \cite{bawu07}, for example. We begin by recalling some basic notions from algebraic number theory.

Let $ \eta $ be an algebraic number of degree $ d $ with minimal primitive polynomial over the integers
$$ a_{0}x^{d}+ a_{1}x^{d-1}+\cdots+a_{d} = a_{0}\prod_{i=1}^{d}(x-\eta^{(i)}),$$
where the leading coefficient $ a_{0} $ is positive and the $ \eta^{(i)} $'s are the conjugates of $ \eta $. Then the \textit{logarithmic height} of $ \eta $ is given by
$$ h(\eta) := \dfrac{1}{d}\left( \log a_{0} + \sum_{i=1}^{d}\log\left(\max\{|\eta^{(i)}|, 1\}\right)\right).$$

In particular, if $ \eta = p/q $ is a rational number with $ \gcd (p,q) = 1 $ and $ q>0 $, then $ h(\eta) = \log\max\{|p|, q\} $. The following are some of the properties of the logarithmic height function $ h(\cdot) $, which will be used in the next sections of this paper without reference:
\begin{eqnarray}
h(\eta\pm \gamma) &\leq& h(\eta) +h(\gamma) +\log 2,\nonumber\\
h(\eta\gamma^{\pm 1})&\leq & h(\eta) + h(\gamma),\\
h(\eta^{s}) &=& |s|h(\eta) ~~~~~~ (s\in\mathbb{Z}). \nonumber
\end{eqnarray}

With the previous notation, Matveev \cite{MatveevII} proved the following theorem, which is our main tool in this paper.
\begin{theorem}\label{Matveev11} Let $ \gamma_{1}, \ldots, \gamma_{t} $ be positive real algebraic numbers in a real algebraic number field $ \mathbb{K} $ of degree $ D $, $ b_{1}, \ldots, b_{t} $ be nonzero integers, and let
$$\Lambda : = \gamma_{1}^{b_{1}}\cdots\gamma_{t}^{b_{t}} - 1,$$
be nonzero. Then
$$\log |\Lambda| > -1.4\times 30^{t+3}\times t^{4.5}\times D^{2}(1+\log D)(1+\log B)A_{1}\cdots A_{t},$$
where
$$B \geq \max\{|b_{1}|, \ldots, |b_{t}|\},$$
and
$$A_{i} \geq \max\{Dh(\gamma_{i}), |\log\gamma_{i}|, 0.16\}, \text{  for all  } i=1, \ldots, t.$$
\end{theorem}

During the course of our calculations, we get some upper bounds on our variables which are too large, thus we need to reduce them. To do so, we use some results from the theory of continued fractions. Specifically, for a nonhomogeneous linear forms in two integer variables, we use a slight variation of a result due to Dujella and Peth\H o \cite{dujella98}, which  itself is a generalization of a result of Baker and Davenport \cite{BD69}.

For a real number $ X $, we write  $ ||X||:= \min\{|X-n|: n\in\mathbb{Z}\} $ for the distance from $ X $ to the nearest integer.
\begin{lemma}\label{Dujjella}
Let $ M $ be a positive integer, $ p/q $ be a convergent of the continued fraction of the irrational number $ \tau $ such that $ q>6M $, and  $ A, B, \mu $ be some real numbers with $ A>0 $ and $ B>1 $. Let further $ \epsilon = ||\mu q||-M||\tau q|| $. If $ \epsilon > 0 $, then there is no solution to the inequality
$$0<|u\tau-v+\mu|<AB^{-w},$$
in positive integers $ u $, $ v $ and $ w $ with
$$ u\leq M \text{  and  } w\geq \dfrac{\log(Aq/\epsilon)}{\log B}.$$
\end{lemma}
Before we conclude this section, we present some useful lemmas that will be used in the next sections on this paper. The following lemma was proved by Bravo and Luca in \cite{BBL17}.
\begin{lemma}\label{fala5}
For $ k\geq 2 $, let $ \alpha $ be the dominant root of $ F^{(k)} $, and consider the function $ f_{k}(z) $ defined in \eqref{fun12}. Then:
\begin{itemize}
\item[(i)]Inequalities
$$\dfrac{1}{2}< f_{k}(\alpha)< \dfrac{3}{4}~~ \text{  and  } ~~|f_{k}(\alpha^{(i)})|<1, ~ 2\leq i\leq k$$
hold. So, the number $ f_{k}(\alpha) $ is not an algebraic integer.
\item[(ii)]The logarithmic height of $f_k(\alpha)$ satisfies $h(f_{k}(\alpha))< 3\log k$.
\end{itemize}
\end{lemma}
Next, we present a useful lemma which is a result due to Cooper and Howard \cite{Howard:2011}.
\begin{lemma}\label{Howard}
For $ k\geq 2 $ and $ n\geq k+2 $,
\begin{eqnarray*}
F_{n}^{(k)}&=&2^{n-2}+\sum_{j=1}^{\left\lfloor \frac{n+k}{k+1}\right\rfloor -1}C_{n,j}2^{n-(k+1)j-2},
\end{eqnarray*}
where
\begin{eqnarray*}
C_{n,j}&=&(-1)^{j}\left[\binom{n-jk}{j}-\binom{n-jk-2}{j-2}\right].
\end{eqnarray*}
\end{lemma}
In the above, we have denoted by $ \lfloor x \rfloor $ the greatest integer less than or equal to $ x $ and also used the convention that $ \binom{a}{b} = 0 $ if either $ a<b $ or if one of $ a $ or $ b $ is negative. In particular, if we assume that $ k+2\leq n\leq 2k+2 $, then $ \lfloor (n+k)/(k+1)\rfloor =2 $, and the formula becomes
\begin{eqnarray}\label{Cooper1}
F_{n}^{(k)}&=&2^{n-2}-(n-k)\cdot 2^{n-k-3}.
\end{eqnarray}

The following estimate was proved by G{\'o}mez and Luca \cite{Gomez}. They used the above result  Lemma \ref{Howard} to prove it.
\begin{lemma}\label{Gomez}
If $ n<2^{k} $, then the following estimates hold:
\begin{eqnarray}\label{Cooper2}
F_{n}^{(k)}&=&2^{n-2}\left(1+\dfrac{k-n}{2^{k+1}}+\dfrac{f(k,n)}{2^{2k+2}}+\zeta(k,n)\right),
\end{eqnarray}
where $ f(k,n)=\frac{1}{2}(z-1)(z+2);~ z=2k-n $ and $ \zeta = \zeta(k,n) $ is a real number such that
\begin{eqnarray*}
|\zeta| < \dfrac{4n^{3}}{2^{3k+3}}.
\end{eqnarray*}
\end{lemma}

\section{Parametric families of solutions}

Assume that $ (n,m)\neq (n_{1}, m_{1}) $ are such that
$$F_{n}^{(k)} - 2^{m} = F_{n_{1}}^{(k)}-2^{m_{1}}.$$
If $ m=m_{1} $, then $ F_{n}^{(k)} = F_{n_{1}}^{(k)} $ and since $ \min\{n,n_{1}\}\geq 2$, we get that $ n=n_{1} $. Thus, $ (n,m) = (n_{1}, m_{1}) $, contradicting our assumption. Hence, $ m\neq m_{1} $, and we may assume without loss of generality that $ m>m_{1}\geq 0 $. Since
\begin{eqnarray}
\label{fala1}
F_{n}^{(k)} - F_{n_{1}}^{(k)} &=& 2^{m}-2^{m_{1}},
\end{eqnarray}
and the right--hand side of \eqref{fala1} is positive, we get that the left--hand side of
\eqref{fala1} is also positive and so $ n>n_{1} $. Thus, since $ F_{1}^{(k)}=F_{2}^{(k)}=1 $,   we may assume that $ n> n_{1}\geq 2 $.

\medskip
We analyze the possible situations.
\medskip

\noindent {\bf Case 1.}
Assume that $ 2\leq n_{1}<n\leq k+1 $. Then, by \eqref{Fibbo111}, we have
\begin{eqnarray*}
F_{n_{1}}^{(k)}=2^{n_{1}-2}\qquad {\rm and} \qquad F_{n}^{(k)}=2^{n-2}
\end{eqnarray*}
so, by substituting in \eqref{fala1}, we get
\begin{eqnarray*}
2^{m}-2^{m_{1}} = 2^{n-2}-2^{n_{1}-2}.
\end{eqnarray*}
The number on the left--hand side of the above equation is $ 2^{m-1}+\cdots+2^{m_{1}} $ and the number on the right--hand side is $ 2^{n-3}+\cdots+2^{n_{1}-2} $. So, by the uniqueness of the binary representation we have $ m=n-2 $ and $ m_{1}=n_{1}-2 $, giving $ c=0 $. All powers of $2$ in the
$k$--generalized Fibonacci sequence are known to be just the numbers $F_s^{(k)}$ with $1\le s\le k+1$ (see \cite{BBL17}). This gives (i) from  the statement of Theorem \ref{Main}.

\medskip

From now on, we assume that $ c\neq 0 $.

\medskip

\noindent {\bf Case 2.}
Assume that $ 2\leq n_{1}\leq k+1 $ and $ k+2\leq n\leq 2k+2 $. Then, by \eqref{Fibbo111} and \eqref{Cooper1}, we have
\begin{eqnarray*}
F_{n_{1}}^{(k)}=2^{n_{1}-2} \qquad {\rm and} \qquad F_{n}^{(k)}=2^{n-2}-(n-k)\cdot 2^{n-k-3}.
\end{eqnarray*}
So, by substituting in \eqref{fala1} as before, we get
\begin{eqnarray}
\label{aux1}
2^{n-2}-2^{n_{1}-2}-(n-k)\cdot 2^{n-k-3} = 2^{m}-2^{m_{1}}.
\end{eqnarray}
In the left--hand side of the above equation, we have
\begin{eqnarray*}
2^{n-2}-2^{n_{1}-2}-(n-k)\cdot 2^{n-k-3}\geq 2^{n-3}-(n-k)\cdot 2^{n-k-3}>2^{n-4}.
\end{eqnarray*}
Indeed the last inequality is equivalent to $ 2^{n-4}> (n-k)\cdot 2^{n-k-3} $, or $ 2^{k-1}>n-k $. Since $ n\leq 2k+2 $, it suffices that $ 2^{k-1}>k+2 $, which indeed holds for all $ k\geq 4 $. Furthermore, unless $ n_{1}=n-1 $, we have
\begin{eqnarray*}
2^{n-2}-2^{n_{1}-2}-(n-k)\cdot 2^{n-k-3}\geq 2^{n-2}-2^{n-4}-(n-k)\cdot 2^{n-k-3}>2^{n-3},
\end{eqnarray*}
from the preceding argument. Thus, we have either $ n_{1} = n-1 $ and then
\begin{eqnarray*}
2^{n-3} \ge 2^{n-2}-2^{n_{1}-2} - (n-k)\cdot 2^{n-k-3}> 2^{n-4},
\end{eqnarray*}
which leads to
\begin{eqnarray*}
2^{n-3} \ge 2^{m}-2^{m_{1}}>2^{n-4},
\end{eqnarray*}
showing that $ m=n-3 $, or $ n_{1}<n-1 $, in which case
\begin{eqnarray*}
2^{n-2}>2^{n-2}-2^{n_{1}-2}-(n-k)\cdot 2^{n-k-3}>2^{n-3},
\end{eqnarray*}
showing that $ m=n-2 $.

We study the two cases. When $ n_{1}=n-1 $, then since $ n_{1}\leq k+1 $, it follows that $ n\leq k+2 $.
Since in fact $ n\geq k+2 $, we get $ n=k+2 $. Then $ m=n-3=k-1 $, so from \eqref{aux1}
\begin{eqnarray*}
2^{k-1}-2^{m_{1}} = 2^{m}-2^{m_{1}}=2^{k}-2^{k-1}-2\cdot 2^{-1}=2^{k-1}-1,
\end{eqnarray*}
showing that $ m_{1}=0 $. So, we have found the parametric family
\begin{eqnarray*}
(n,m,n_{1}, m_{1}) = (k+2, k-1, k+1, 0)
\end{eqnarray*}
for which $c= 2^{k-1}-1$ according to \eqref{Problem}. This corresponds to situation (ii--$a$) in the statement of Theorem \ref{Main}.

A different possibility is $ n_{1}<n-1 $, in which case $ m=n-2 $. Now \eqref{aux1} leads to
\begin{eqnarray*}
2^{n-2}-2^{n_{1}-2}-(n-k)\cdot 2^{n-k-3}=2^{n-2}-2^{m_{1}}
\end{eqnarray*}
so
\begin{eqnarray*}
(n-k)\cdot 2^{n-k-3}=2^{m_{1}}-2^{n_{1}-2}.
\end{eqnarray*}
Simplifying the powers of $ 2 $, we get
$$
n-k=2^{m_1-(n-k)+3}-2^{n_1-(n-k)+1}.
$$
Thus, $n-k\in [2,k+2]$ is a difference of two powers of $2$. Take any number in $[2,k+2]$ which is a difference of two powers of $2$. Let it be $2^a-2^b$. Note that $a>b$ and $(a,b)\ne (1,0)$. Set
$$
n-k=2^a-2^b.
$$
This gives $n=k+2^a-2^b\in [k+2,2k+2]$. Next we have $n_1-(n-k)+1=b$. Then $n_1=b+(n-k)-1$. But $n_1\le k+1$. This gives $(b-1)+(n-k)\le k+1$, so $(b-1)+2^a-2^b\le k+1$. But we started with $2^a-2^b\in [2,k+2]$. So, in fact we get
$$
b+2^a-2^b\le k+2
$$
and $2^a-2^b\ge 2$. If $n-k=2$, then $(a,b)=(2,1)$, otherwise $n-k\ge 3$ and $b\ge 0$. Finally, $m_1+3-(n-k)=a$. Thus,
$$m_1=(a-3)+(n-k)=(a-1)+((n-k)-2)$$ and this is nonnegative
from the preceding discussion. So, the family is
$$
(n,m,n_1,m_1)=(k+2^a-2^b, k+2^a-2^b-2, b-1+2^a-2^b,a-3+2^a-2^b),
$$
where $(a,b)$ are such that $a>b\ge 0$, $(a,b)\ne (1,0)$, and $b+2^a-2^b\le k+2$. Furthermore, by \eqref{Problem}, we have
$c = 2^{b-3+2^a-2^b}-2^{a-3+2^a-2^b}$. This corresponds to situation (ii--$b$) in the statement of Theorem  \ref{Main}.

\medskip
\noindent {\bf Case 3.}
Assume that $k+2\le n_1<n\le 2k+2$. Then, by \eqref{Cooper1}, we have that
$$
F_{n_1}^{(k)}=2^{n_1-2}-(n_1-k) 2^{n_1-k-3} \qquad {\rm and} \qquad F_n^{(k)}=2^{n-2}-(n-k)2^{n-k-3}.
$$
Then by a similar substitution as before, equation \eqref{fala1}
translates into
\begin{eqnarray}
\label{aux2}
2^{n-2}-2^{n_1-2}-\left((n-k)\cdot 2^{n-k-3}-(n_1-k)\cdot 2^{n_1-k-3}\right)=2^m-2^{m_1}.
\end{eqnarray}
Since $n_1\leq n-1$, the left--hand side is at least
\begin{eqnarray*}
2^{n-2}-2^{n_1-2} &-& ((n-k)\cdot 2^{n-k-3}-(n_1-k)\cdot2^{n_1-k-3})\\
&\ge& 2^{n-2}-2^{n-3} -  ((n-k)\cdot 2^{n-k-3}-(n-k-1)\cdot 2^{n-k-4})\\
& = & 2^{n-3}-(n-k+1)\cdot 2^{n-k-4}>2^{n-4},
\end{eqnarray*}
where the last inequality is equivalent to
$$
2^{n-4}>(n-k+1)\cdot 2^{n-k-4},
$$
or
$$
2^k>n-k+1.
$$
Since $n-k\le k+2$, it suffices that $2^k>k+2+1=k+3$, which holds for $k\ge 4$. Thus, if $n_1=n-1$, then
$$
2^{n-3}>2^{n-2}-2^{n-3}-\left((n-k)\cdot 2^{n-k-3}-(n-k-1)\cdot 2^{n-k-4}\right)>2^{n-4},
$$
so
$$
2^{n-3}>2^m-2^{m_1}>2^{n-4},
$$
giving $m=n-3$. In this case, we get from \eqref{aux2},
$$
2^{n-2}-2^{n-3}-(n-k+1)\cdot 2^{n-k-4}=2^{n-3}-2^{m_1},
$$
so
$$
(n-k+1)\cdot 2^{n-k-4}=2^{m_1},
$$
giving
$$
n-k+1=2^{m_1+4-(n-k)}.
$$
Thus, $n-k+1=2^t$ is a power of two in the interval $[3,k+3]$ (so $t\ge 2$). Further, $n=2^t+k-1,~n_1=n-1=2^t+k-2,~m=n-3=2^t+k-4$ and $m_1=n-k-4+t=2^t+t-5$.
Since $t\ge 2$, we get that $m_1>0$. Hence,
\begin{eqnarray*}
(n, m, n_{1}, m_{1}) = (k+2^t-1, k+2^t-4, k+2^t-2, t+2^t-5)
\end{eqnarray*}
which corresponds to the parametric family (iii), with $c = 2^{k+2^t-4}+2^{2^t-4}-2^{t+2^t-4}$, in the statement of the Theorem \ref{Main}.
\medskip

Next we consider the situation $n_1<n-1$. We show that there are no solutions in this case. Then,
\begin{eqnarray*}
2^{n-2} & > & 2^{n-2}-2^{n_1-2}-((n-k)\cdot 2^{n-k-3}-(n_1-k)\cdot 2^{n_1-k-3})\\
& \ge &
2^{n-2}-2^{n-4}-\left((n-k)\cdot 2^{n-k-3}-(n-k-2)\cdot 2^{n-k-5}\right)\\
& > & 2^{n-3}.
\end{eqnarray*}
The last inequality is equivalent to
$$
2^{n-4}>(n-k)\cdot 2^{n-k-3}-(n-k-2)\cdot 2^{n-k-5},
$$
which is implied by
$$
2^{n-4}>(n-k)\cdot 2^{n-k-3},
$$
or
$$
2^{k-1}>n-k.
$$
Since $n-k\le k+2$, it suffices that $2^{k-1}> k+2$ and this holds for all $k\ge 4$. Thus, for $n_1<n-1$, we have
$$
2^{n-2}>F_n^{(k)}-F_{n_1}^{(k)}>2^{n-3},
$$
so
$$
2^{n-2}>2^m-2^{m_1}>2^{n-3},
$$
showing that $m=n-2$. In this case, we have by \eqref{aux2}, that
$$
2^{n-2}-2^{n_1-2}-(n-k)\cdot 2^{n-k-3}+(n_1-k)\cdot 2^{n_1-k-3}=2^{n-2}-2^{m_1},
$$
giving
$$
(n-k)\cdot 2^{n-k-3}-(n_1-k)\cdot 2^{n_1-k-3}=2^{m_1}-2^{n_1-2}.
$$
The left--hand side is positive therefore so is the right--hand side. Thus,
\begin{equation}\label{eq:222}
2^{n_1-k-3}(2^{n-n_1}(n-k)-(n_{1}-k))=2^{n_1-2}(2^{m_1-n_1+2}-1).
\end{equation}

To proceed, we write
$$
n-k=2^{\alpha} u\qquad {\text{\rm and}}\qquad n_1-k=2^{\alpha_1} u_1,
$$
where $\alpha,~\alpha_1$ are nonnegative and $u,~u_1$ are odd. Since $n-k,~n_1-k\in [2,k+2]$, it follows $2^{\alpha}\le k+2$ and $2^{\alpha_1}\le k+2$. Hence,
$\max\{\alpha,\alpha_1\}\le \log(k+2)/\log 2$. Equation \eqref{eq:222} becomes
\begin{equation}
\label{eq:2222}
2^{n_1-k-3}(2^{\alpha+n-n_1} u-2^{\alpha_1} u_1)=2^{n_1-2}(2^{m_1-n_1+2}-1).
\end{equation}

We distinguish various cases.

\medskip

\noindent{\bf Case 3.1} $\alpha+n-n_1=\alpha_1$.
In this case, by \eqref{eq:2222}, we have
\begin{equation}
\label{eq:u=u1}
2^{n_1-k-3+\alpha_1}(u-u_1)=2^{n_1-2}(2^{m_1-n_1+2}-1).
\end{equation}
Note that we cannot have $u=u_1$ (otherwise we get $n=n_1$, a contradiction). Since the exponent of $2$ in the right in \eqref{eq:u=u1} is exactly $n_1-2$ and in the left is at least $n_1-k-3+\alpha_1$, we get that $n_1-2\ge n_1-k-3+\alpha_1$, so $k+1\ge \alpha_1$, and
$$
u-u_1=2^{k+1-\alpha_1}(2^{m_1-n_1+2}-1).
$$
We deduce that the following inequality holds:
$$
 2^{k+1-\alpha_1}\mid u-u_1,\quad {\text{\rm so}}\quad k+1-\alpha_1\le \frac{\log(u-u_1)}{\log 2}\le \frac{\log (k+1)}{\log 2}.
$$
Thus,
$$
k+1  =  (k+1-\alpha_1)+\alpha_1\le \frac{\log(k+1)}{\log 2}+\frac{\log(k+2)}{\log 2},
$$
which yields
$$
2^{k+1}\le (k+2)(k+1),
$$
so $k\le 3$. So, this case cannot lead to infinitely many solutions.

\medskip

\noindent{\bf Case 3.2} $\alpha+n-n_1<\alpha_1$.
In this case, by \eqref{eq:2222}, we now have
$$
2^{n-k-3+\alpha}(u-2^{\alpha_1-\alpha-n+n_1} u_1)=2^{n_1-2} (2^{m_1-n_1+2}-1).
$$
Identifying factors which are powers of $2$ in both sides, we have
$$
n_1=n+\alpha-k-1.
$$
Since
$$
n-n_1< \alpha_1-\alpha\le \alpha_1\le \frac{\log(k+2)}{\log 2},
$$
we have
$$
k+1  =  (n-n_1)+\alpha\le  \frac{\log(k+2)}{\log 2}+\frac{\log(k+2)}{\log 2},
$$
giving
$$
2^{k+1}\le (k+2)^2,
$$
so $k\le 4$.

Thus, as in the previous case, this situation cannot lead us to infinitely many solutions either.

\medskip

\noindent{\bf Case 3.3} $\alpha_1<\alpha+n-n_1$.
\medskip
In this case, \eqref{eq:2222} becomes
$$
2^{n_1-k-3+\alpha_1}(2^{\alpha-\alpha_1+n-n_1} u-u_1)=2^{n_1-2}(2^{m_1-n_1+2}-1).
$$
Identifying powers of $2$ in both sides above, we get
$$
k+1=\alpha_1.
$$
Hence,
$$
k+1\le \frac{\log (k+2)}{\log 2},
$$
giving $2^{k+1}\le k+2$, so $k\le 1$, a contradiction.

The last parametric family from the statement of Theorem \ref{Main} will be identified in the next section.

\section{Solutions with $n\ge 2k+3$}

\medskip
From now on, we searched for solutions other than the ones given in Theorem \ref{Main} (i), (ii), and (iii), with the aim is to show that perhaps they are none except for some sporadic ones with $k<k_0$ with some small $k_0$. Then the problem will be solved by finding individually for every $k\in [4,k_0]$, the values of $c$ such that \eqref{Problem} has some solution $(n,m,n_1,m_1)$ with $n>n_1, ~m>m_1$ and determining for each $c$ all such representations. It turns out that this program does not quite work out since along the way we find parametric family (iv) with $n=2k+3$, but afterwards all does work out and we are able to show that indeed if $n>2k+3$, then $k\le 790$.

So, let's get to work. We go back to \eqref{Problem} and assume that $n\ge 2k+3$. Suppose first that $m\ge n-1$. We recall equality \eqref{fala1}:
$$
2^m-2^{m_1}=F_n^{(k)}-F_{n_1}^{(k)}.
$$
The left--hand side is positive and
$$
2^m-2^{m_1}\ge 2^{m-1}\ge 2^{n-2}>F_n^{(k)}>F_n^{(k)}-F_{n_1}^{(k)},
$$
where we used the fact that $F_n^{(k)} < 2^{n-2}$ for $n\ge k+2$. Thus,
$m\le n-2$. Note that $n\ge 2k+3$, so $n-2k\ge 3$.

We put $y:=n/2^k$, and assume that
\begin{equation}
\label{hip}
n^3 < 2^{k-5}, \qquad {\rm so } \qquad \quad y<1/4.
\end{equation}
Thus, by Lemma \ref{Gomez},
we have
$$
F_n^{(k)}=2^{n-2}(1-\zeta), \qquad {\rm where} \qquad |\zeta|<\frac{1}{2}\left(y+y^2+y^3\right).
$$
Similarly,
\begin{equation}
\label{eq:Fn1}
F_{n_1}^{(k)}=2^{n_1-2}(1-\zeta_1), \qquad {\rm where ~~ also} \qquad  |\zeta_1|<\frac{1}{2}(y+y^2+y^3).
\end{equation}

We get from \eqref{fala1}
\begin{equation}
\label{eq:10}
\left|(2^m-2^{m_1}) - (2^{n-2}-2^{n_1-2})\right|<(2^{n-2}+2^{n_1-2})\frac{(y+y^2+y^3)}{2}<2^{n-2} y.
\end{equation}
If $m\le n-4$, then the left--hand side in \eqref{eq:10} is at least
$$
(2^{n-2}-2^{n_1-2})-2^{n-4}\ge 2^{n-3}-2^{n-4}\ge 2^{n-4},
$$
showing that
$$
2^{n-4}\le 2^{n-2} y,
$$
giving $y\ge1/4$, a contradiction to \eqref{hip}. Further, assuming that $m=n-3$ but $n_1<n-1$, the left--hand side in formula \eqref{eq:10} is at least
$$
(2^{n-2}-2^{n_1-2})-2^m\ge 2^{n-2}-2^{n-4}-2^{n-3}=2^{n-4},
$$
and we get to the same contradiction to \eqref{hip}, namely that $y\ge 1/4$. Thus, we conclude that either $(m,n_1)=(n-3,n-1)$, or $m=n-2$. The first case gives from \eqref{fala1}
\begin{equation}
\label{eq:13}
F_{n}^{(k)}-F_{n-1}^{(k)}=2^{n-3}-2^{m_1}.
\end{equation}
Using Lemma \ref{Gomez}, we get
\begin{equation}
\label{eq:Fn-aux2}
F_n^{(k)}=2^{n-2}\left(1-\frac{n-k}{2^{k+1}}+\eta\right)\quad {\rm and} \quad F_{n-1}^{(k)}=2^{n-3}\left(1-\frac{n-k-1}{2^{k+1}}+\eta_1\right),
\end{equation}
where
$$
\max\{|\eta|,|\eta_1|\} \le \frac{1}{2}(y^2+y^3)<y^2.
$$
Putting these into \eqref{eq:13}, we get
$$
\left|-2^{n-k-3}(n-k)+2^{n-k-4}(n-k-1)+2^{m_1}\right|<2^{n-2}|\eta|+2^{n-3}|\eta_1|<2^{n-1}y^2.
$$
In the left--hand side, we have the amount
$$
|2^{m_1}-2^{n-k-4}(n-k+1)|.
$$
If $m_1\le n-k-4$, then this amount exceeds $2^{n-k-4}(n-k)>2^{n-k-4}$. If $m_1> n-k-4$, then the above number can be rewritten as
$$
2^{n-k-4}|n-k+1-2^{m_1-(n-k-4)}|.
$$
If $n-k+1 \neq 2^{m_1-(n-k-4)}$, then the above amount is $\ge 2^{n-k-4}$. We thus get in all the above instances
$$
2^{n-k-4}\le |2^{m_1}-2^{n-k-4}(n-k+1)|<2^{n-1} y^2<\frac{2^{n-1}n^2}{2^{2k}},
$$
giving
$$
n^2>2^{k-3}\qquad {\rm so} \qquad n > 2^{(k-3)/2},
$$
a contradiction to \eqref{hip}. If $n-k+1 = 2^{m_1-(n-k-4)}$, we consider one more term in \eqref{eq:Fn-aux2}:
\begin{eqnarray*}
\label{eq:Fn-aux3}
F_n^{(k)} & = & 2^{n-2}-2^{n-k-3}(n-k)+2^{n-2k-5}(n-2k+1)(n-2k-2)+2^{n-2}\delta,\\
F_{n_1}^{(k)} & = & 2^{n-3}-2^{n-k-4}(n-k-1)+2^{n-2k-6}(n-2k)(n-2k-3)+2^{n-3}\delta_1
\end{eqnarray*}
where
$$
2^{n-2}|\delta| < 2^{n-3} y^3<2^{n-3k-3}n^3 \quad {\rm \and} \quad 2^{n-3}|\delta_1|<2^{n-4} y^3 < 2^{n-3k-4}n^3.
$$
Thus, by \eqref{hip},
\begin{eqnarray}
\label{ineq1}
\max\{2^{n-2}|\delta|, 2^{n-3}|\delta_1|\} < 2^{n-2k-8}.
\end{eqnarray}
Putting these into \eqref{eq:13}, we get
$$
2^{n-2k-6}\left|2(n-2k+1)(n-2k-2) - (n-2k)(n-2k-3)\right| < 2^{n-2}|\delta|+2^{n-3}|\delta_1| < 2^{n-2k-7}.
$$
Taking $w:=n-2k$, we have that
\begin{eqnarray*}
\label{eq:Fn-aux3}
 w^2+w-4  = |2(w+1)(w-2)-w(w-3)| < 1/2,
\end{eqnarray*}
which is a contradiction for all $k\ge 4$, given that $n\ge2k+3$. So, the situation $(m,n_1)=(n-3,n-1)$ is not possible.

Hence, we continue with the case $m=n-2$. Going back to \eqref{Problem}, we have
\begin{equation}
\label{eq:20}
F_n^{(k)}-2^{n-2}=F_{n_1}^{(k)}-2^{m_1}.
\end{equation}
The number on the left--hand side in \eqref{eq:20} is negative. We will show that $m_1\ge n_1-2$. Indeed, suppose that $m_1\le n_1-3$. Since for us $y<1/4$, we get $|\zeta_1|<1/2$ (see \ref{eq:Fn1}). Further, again by \eqref{eq:Fn1}, we note that $F_{n_1}^{(k)}>2^{n_1-3}\ge 2^{m_1}$, so the
right--hand side in \eqref{eq:20} is positive, a contradiction. Thus, $m_1\ge n_1-2$. The case $m_1=n_1-2$ leads to
\begin{equation}
\label{eq:21}
F_n^{(k)}-2^{n-2}=F_{n_1}^{(k)}-2^{n_1-2}.
\end{equation}
Since $c\ne 0$, it follows that $n_1\ge k+2$. However, we have the following lemma.

\begin{lemma}
The sequence $\{2^{n-2}-F_n^{(k)}\}_{n\ge k+2}$ is increasing for $n\ge k+3$.
\end{lemma}

\begin{proof}
We want
$$
2^{n-1}-F_{n+1}^{(k)}>2^{n-2}-F_n^{(k)},
$$
which is equivalent to
$$
2^{n-2}>F_{n-1}^{(k)}+F_{n-2}^{(k)}+\cdots+F_{n-k+1}^{(k)}.
$$
There are $k-1$ terms in the right--hand side. Each of them satisfies $F_{n-1-j}^{(k)}\le 2^{n-3-j}$ for $j=0,1,\ldots,k-2$ because $F_a\le 2^{a-2}$ holds for all $a\ge 2$. Thus, it suffices that
$$
2^{n-2}>2^{n-3}+2^{n-4}+\cdots+2^{n-k-1},
$$
which is obvious.
\end{proof}

Thus, \eqref{eq:21} is impossible. Hence, $m_1\ge n_1-1$.  Using the first identity in \eqref{eq:Fn-aux2}, we have that the left--hand side in \eqref{eq:20} is
\begin{equation}
\label{eq*}
-2^{n-k-3}(n-k)+2^{n-2}\eta,
\end{equation}
where $|\eta|<(y^2+y^3)/2<y^2$. Note that since
$$
y^2=\frac{n^2}{2^{2k}}<\frac{1}{2^{k+3}}\qquad ({\text{\rm by}}~~\ref{hip}),
$$
it follows that
\begin{equation}
\label{eq:eta}
2^{n-2}|\eta|<2^{n-k-5}.
\end{equation}
Thus, the left--hand side of \eqref{eq:20} is in the interval
$$
(-2^{n-k-3}(n-k+1/2),-2^{n-k-3}(n-k-1/2)).
$$
Now the right--hand side of \eqref{eq:20} is in the interval $(-2^{m_1}, -2^{m_1-1}]$, where for the right--hand extreme of the interval we used the fact that $F_{n_1}^{(k)}\le 2^{n_1-2}\le 2^{m_1-1}$. Comparing them we get
$$
-2^{n-k-3}(n-k+1/2)<-2^{m_1-1}\qquad {\text{\rm and}}\qquad -2^{n-k-3}(n-k-1/2)>-2^{m_1},
$$
which gives
$$
2^{m_1-1}<2^{n-k-3}(n-k+1/2)\qquad {\text{\rm and}}\qquad 2^{n-k-3}(n-k-1/2)<2^{m_1}.
$$
In particular, $m_1\ge n-k-3$, so
$$
2^{m_1-(n-k-3)-1}\le n-k\le 2^{m_1-(n-k-3)}.
$$
We thus get, from \eqref{eq:20} and \eqref{eq*}, that
\begin{equation}
\label{eq:30}
-2^{n-k-3}(n-k-2^{m_1-(n-k-3)})=F_{n_1}^{(k)}-2^{n-2}\eta.
\end{equation}

We distinguish two cases.

\medskip
\noindent{\bf Case 1.} Assume that $n_1<n-k-1$.
\medskip

Then $F_{n_1}^{(k)}<2^{n_1-2}\le 2^{n-k-4}$. Using also \eqref{eq:eta},   we get
$$
2^{n-k-3}\left|(n-k)-2^{m_1-(n-k-3)}\right|<\max\{F_{n_1}^{(k)},2^{n-2}|\eta|\}<2^{n-k-4},
$$
so $n-k-2^{m_1-(n-k-3)}$ is an integer which is at most $1/2$ in absolute value. Hence, it is zero. Thus, $n-k=2^{m_1-(n-k-3)}$.
We now go one more step and say that
\begin{eqnarray*}
F_n^{(k)} & = & 2^{n-2}-2^{n-k-3}(n-k)+2^{n-2k-5}(n-2k+1)(n-2k-2)+2^{n-2}\delta,\\
F_{n_1}^{(k)} & = & 2^{n_1-2}(1-\eta_1),\nonumber
\end{eqnarray*}
where, by \eqref{ineq1},
$$
2^{n-2}|\delta|<2^{n-2k-8}.
$$
Further, by \eqref{eq:Fn1},
$$
2^{n_1-2} |\eta_1|<2^{n_1-3} y^2<2^{n-3k-4} n^2<2^{n-2k-8}.
$$
Equation \eqref{eq:20} now implies that
$$
-2^{n-k-3}(n-k)+2^{n-2k-5}(n-2k+1)(n-2k-2)+2^{n-2}\delta = 2^{n_1-2} - 2^{n_1-2}\eta_1 - 2^{m_1},
$$
so, given that $n-k = 2^{m_1-(n-k-3)}$,
\begin{equation}
\label{eq:111}
2^{n-2k-5}(n-2k+1)(n-2k-2)-2^{n_1-2}=-2^{n-2}\delta - 2^{n_1-2}\eta_1.
\end{equation}
Assume that $n_1\le n-2k-4$. Then
\begin{eqnarray*}
2^{n-2k-5}\le 2^{n-2k-5}(n-2k+1)(n-2k-2) & \le & 2^{n_1-2}+2^{n-2}|\delta|+2^{n_1-2}|\eta_1|\\
& < & 3\times 2^{n-2k-7}<2^{n-2k-5},
\end{eqnarray*}
which is a contradiction. Thus, we must have $n_1 \ge n-2k-3$, so
\begin{eqnarray*}
 2^{n-2k-5}\left|(n-2k+1)(n-2k-2)-2^{n_1-2-(n-2k-5)}\right| & \le & 2^{n-2}|\delta|+2^{n_1-2}|\eta_1|\\
& < & 2^{n-2k-7}.
\end{eqnarray*}
The left--hand side above is an integer divisible by $2^{n-2k-5}$. Since it is smaller than $2^{n-2k-7}$, it must be the zero integer. Thus, with $w=n-2k$, we have
$$
(w+1)(w-2)=2^{n_1-2-(n-2k-5)}.
$$
In the  left--hand side above, one of the factors $w-2$ and $w+1$ is odd. Since they are both positive and powers of $2$, it follows that the smaller one is $1$. Hence, $w=3$, so
$$
w+1 = 2^2=2^{n_1-2-(n-2k-5)},
$$
giving $n_1-2 = n-2k-3=0$. Thus, $n = 2k+3, n_1 = 2$ and $m = n-2 = 2k+1$. From equality \eqref{eq:20}, we conclude that
$$
2^{2k+1}-2^k(k+3) = 2^{2k+1}- 2^{m_1}
$$
giving $k+3 = 2^t,$ for some integer $t\ge3$ and $m_1 = k+t$. Hence, we obtain the parametric family 
$$
(n, m, n_1, m_1) = (2^{t+1}-3, 2^{t+1}-5, 2, t+2^t-3 )
$$
with $c = 1 - 2^{t+2^t-3}$, which corresponds to situation (iv) in the statement of
the Theorem \ref{Main}.

\medskip
\noindent{\bf Case 2.} $n_1\ge n-k-1$.
\medskip

The equation that we then get from \eqref{eq:Fn1} and \eqref{eq:30} is
$$
2^{n-k-3}\left((n-k)-2^{m_1-(n-k-3)}+2^{n_1-2-(n-k-3)}\right)=2^{n_1-2}\eta_1 + 2^{n-2}\eta.
$$
Given that $m_1<m$ and that we are in the case $m = n-2$, we have
$$
2^{n_1-2} |\eta_1|\le 2^{m_1-1} |\eta_1|\le 2^{n-k-3}(n-k)y<2^{n-2k-3} n^2<2^{n-k-7}.
$$
We thus have
$$
2^{n-k-3}|(n-k)-2^{m_1-(n-k-3)}+2^{n_1-2-(n-k-3)}|<2^{n_1-2}|\eta_1|+2^{n-2}|\eta|<2^{n-k-5},
$$
showing the left--hand side is zero. Thus, $a=m_1-(n-k-3)$, $b=n_1-2-(n-k-3)$ and
$$
n-k=2^a-2^b.
$$
So, $n=k+2^a-2^b$. As in previous iterations, we go one step further and write
\begin{eqnarray*}
F_n^{(k)} & = & 2^{n-2}-2^{n-k-3}(n-k)+2^{n-2k-5}(n-2k+1)(n-2k-2)+2^{n-2}\delta,\\
F_{n_1}^{(k)} & = & 2^{n_1-2}-2^{n_1-k-3}(n_1-k)+2^{n_1-2} \eta_1.
\end{eqnarray*}
Inserting these into equation \eqref{eq:20}, we get
\begin{eqnarray*}
-2^{n-k-3}(n-k) & + & 2^{n-2k-5}(n-2k-2)(n-2k+1)+2^{n-2}\delta\\
& = & 2^{n_1-2}-2^{n_1-k-3}(n_1-k)-2^{m_1}+2^{n_1-2}\eta_1,
\end{eqnarray*}
or
$$
2^{n-2k-5}(n-2k-2)(n-2k+1) + 2^{n_1-k-3}(n_1-k)=-2^{n-2}\delta+2^{n_1-2}\eta_1.
$$
We have $n_1=n-k-1+b$, so $n_1-k-3=n-2k-4+b$ so $n_1-k-3-(n-2k-5)=b+1$ and $n_1-k=n-2k-1+b$. Thus,
$$
2^{n-2k-5}\left((n-2k+1)(n-2k-2) + 2^{b+1}(n-2k-1+b)\right)=-2^{n-2}\delta+2^{n_1-2}\eta_1.
$$
We already know that $2^{n-2}|\delta|<2^{n-2k-8}$. Now $n_1-2=n-k-3+b$, so
$$
2^{n_1-2}=2^{n-k-3} 2^b.
$$
Note that $n-k=2^a-2^{b}\ge 2^{a-1}\ge 2^b$, so $2^b<n$. Thus,
$$
2^{n_1-2} |\eta_1|\le 2^{n-k-3} n y^2\le 2^{n-k-3} n^3/2^{2k}<2^{n-2k-8},
$$
since $n^3<2^{k-5}.$ Hence, we get
$$
2^{n-2k-5}|(n-2k+1)(n-2k-2) + 2^{b+1}(n-2k-1+b)|<2^{n-2k-7},
$$
showing that the number in absolute value is zero, which is a contradiction because $n-2k\ge3$ and $b\ge0$.
In conclusion, there are no solutions with $n>2k+3$ provided that \eqref{hip} holds. 
In the next section, we estimate a value of $k_0$ for which inequality \eqref{hip} is fulfilled for all $k> k_0$.

\section{Establishing an inequality in terms of $ n $ and $ k $ and estimating $k_0$}

Since $ n>n_{1}\geq 2 $, we have that $ F_{n_{1}}^{(k)}  \leq F_{n-1}^{(k)}$ and therefore
$$F_{n}^{(k)} = F_{n-1}^{(k)} + \cdots + F_{n-k}^{(k)} \geq F_{n-1}^{(k)} + \cdots + F_{n-k-1}^{(k)} \geq F_{n_{1}}^{(k)} + \cdots + F_{n-k-1}^{(k)}.$$
So, from the above, \eqref{Fib12} and \eqref{fala1}, we have
\begin{eqnarray}\label{fala4}
\alpha^{n-4}&\leq& F_{n-2}^{(k)}\leq F_{n}^{(k)}-F_{n_{1}}^{(k)} = 2^{m}-2^{m_{1}} < 2^{m}, \text{  and  }\\
\alpha^{n-1}&\geq& F_{n}^{(k)} > F_{n}^{(k)}-F_{n_{1}}^{(k)} = 2^{m}-2^{m_{1}}\geq 2^{m-1},\nonumber
\end{eqnarray}
leading to
\begin{eqnarray}\label{fala2}
1+\left(\dfrac{\log 2}{\log\alpha}\right) (m - 1) < n< \left(\dfrac{\log 2}{\log \alpha}\right)m + 4.
\end{eqnarray}

We note that the above inequality \eqref{fala2} in particular implies that $ m < n< 1.2m+4 $. Moreover, note that we can assume $ n\geq k+2 $, since otherwise, this would give us only the solution for $ c = 0$, which is family (i) of Theorem \ref{Main}.

We assume for technical reasons that $ n>1600 $. By \eqref{approxgap} and \eqref{fala1}, we get
\begin{eqnarray*}
\left|f_{k}(\alpha)\alpha^{n-1} - 2^{m}\right|&=&\left|(f_{k}(\alpha)\alpha^{n-1}-F_{n}^{(k)})+(F_{n_{1}}^{(k)}-2^{m_{1}})\right|\\
&=&\left|(f_{k}(\alpha)\alpha^{n-1}-F_{n}^{(k)})+(F_{n_{1}}^{(k)}-f_{k}(\alpha)\alpha^{n_{1}-1})+(f_{k}(\alpha)\alpha^{n_{1}-1}-2^{m_{1}})\right|\\
&<&\dfrac{1}{2}+\dfrac{1}{2}+\alpha^{n_{1}-1}+2^{m_{1}}\\
&<& \alpha^{n_{1}}+2^{m_{1}}\\
&<& 2\max\{\alpha^{n_{1}}, 2^{m_{1}}\}.
\end{eqnarray*}
In the above, we have also used the fact that $ |f_{k}(\alpha)| < 1 $. Dividing through by $ 2^{m} $ we get
\begin{eqnarray}\label{fala3}
&&\left|f_{k}(\alpha)\alpha^{n-1}2^{-m} -1\right|< 2\max\left\{\dfrac{\alpha^{n_{1}}}{2^{m}}, 2^{m_{1}-m}\right\} < \max\{\alpha^{n_{1}-n+6}, 2^{m_{1}-m+1}\},
\end{eqnarray}
where for the right--most inequality in \eqref{fala3} we used \eqref{fala4} and the fact that $ \alpha^{2}> 2 $.

For the left-hand side of \eqref{fala3} above, we apply Theorem \ref{Matveev11} with the data
$$
t:=3, ~~\gamma_{1}:=f_{k}(\alpha), ~~\gamma_{2}: = \alpha,~~\gamma_{3}:=2, ~~ b_{1}:=1, ~~b_{2}:=n-1, ~~b_{3}:=-m .
$$
We begin by noticing that the three numbers $ \gamma_{1}, \gamma_{2}, \gamma_{3} $ are positive real numbers and belong to the field $ \mathbb{K}: = \mathbb{Q}(\alpha)$, so we can take $ D:= [\mathbb{K}:\mathbb{Q}]= k$. Put
$$\Lambda :=f_{k}(\alpha)\alpha^{n-1}2^{-m} -1. $$
To see why $ \Lambda \neq 0 $, note that otherwise, we would then have that $ f_{k}(\alpha) = 2^{m}\alpha^{-(n-1)} $ and so $ f_{k}(\alpha) $ would be an algebraic integer, which contradicts Lemma \ref{fala5} (i).

Since $ h(\gamma_{2})= (\log \alpha)/k <(\log 2)/k $ and $ h(\gamma_{3})= \log 2$, it follows that we can take $ A_{2}:= \log 2$ and $ A_{3}:= k\log 2 $. Further, in view of Lemma \ref{fala5} (ii), we have that $ h(\gamma_{1})<3\log k$, so we can take $ A_{1}:=3k\log k $. Finally, since $ \max\{1, n-1, m\} = n-1$, we take $ B:=n $.

Then, the left--hand side of \eqref{fala3} is bounded below, by Theorem \ref{Matveev11}, as
$$\log |\Lambda| >-1.4\times 30^6 \times 3^{4.5} \times k^4 (1+\log k)(1+\log n)(3\log k)(\log 2)(\log 2).$$
Comparing with \eqref{fala3}, we get
$$\min\{(n-n_{1}-6)\log\alpha , (m-m_{1}-1)\log 2\} < 4.2\times 10^{11} k^4 \log^{2}k(1+\log n),$$
which gives
$$\min\{(n-n_{1})\log\alpha , (m-m_{1})\log 2\} < 4.25\times 10^{11} k^4 \log^{2}k(1+\log n).$$
Now the argument is split into two cases.\\

\textbf{Case 1.} $\min \lbrace (n-n_1) \log \alpha , (m-m_1) \log 2 \rbrace  = (n - n_{1}) \log \alpha$.

\medskip

In this case, we rewrite \eqref{fala1} as
\begin{eqnarray*}
\left| f_{k}(\alpha)\alpha^{n-1} - f_{k}(\alpha)\alpha^{n_{1}-1} - 2^{m}\right| &=& \left|(f_{k}(\alpha)\alpha^{n-1}- F_{n}^{(k)}) + (F_{n_{1}}^{(k)} - f_{k}(\alpha)\alpha^{n_{1}-1}) - 2^{m_{1}}\right|\\
&<&\dfrac{1}{2}+\dfrac{1}{2}+2^{m_{1}} \leq 2^{m_{1}+1}.
\end{eqnarray*}
Dividing through by $ 2^{m} $ gives
\begin{eqnarray}\label{fala6}
\left| f_{k}(\alpha)(\alpha^{n-n_{1}}-1)\alpha^{n_{1}-1}2^{-m} - 1\right|&<&2^{m_{1}-m+1}.
\end{eqnarray}
Now we put
$$
\Lambda_{1} := f_{k}(\alpha)(\alpha^{n-n_{1}}-1)\alpha^{n_{1}-1}2^{-m} - 1.
$$
We apply again Theorem \ref{Matveev11} with the following data
$$
t:=3,~\gamma_{1} :=f_{k}(\alpha)(\alpha^{n-n_{1}}-1), ~ \gamma_{2}:=\alpha, ~\gamma_{3}:=2,  ~~ b_{1}:=1, ~b_{2}:=n_{1}-1, ~b_{3}:=-m.
$$
As before, we begin by noticing that the three numbers $ \gamma_{1}, \gamma_{2}, \gamma_{3} $ belong to the field $ \mathbb{K} := \mathbb{Q}(\alpha) $, so we can take $ D:= [\mathbb{K}: \mathbb{Q}] = k$. To see why $ \Lambda_{1} \neq 0$, note that otherwise, we would get the relation $ f_{k}(\alpha)(\alpha^{n-n_{1}}-1) = 2^{m}\alpha^{1-n_{1}} $. Conjugating this last equation with any automorphism $ \sigma$ of the Galois group of $ \Psi_{k}(x) $ over $ \mathbb{Q} $ such that $ \sigma(\alpha) = \alpha^{(i)} $ for some $ i\geq 2 $, and then taking absolute values, we arrive at the equality $ |f_{k}(\alpha^{(i)})((\alpha^{(i)})^{n-n_1}-1)| = |2^{m}(\alpha^{(i)})^{1-n_1}| $. But this cannot hold because, $ |f_{k}(\alpha^{(i)})||(\alpha^{(i)})^{n-n_1}-1|<2 $ since $ |f_{k}(\alpha^{(i)})|<1 $ by Lemma \ref{fala5} (i), and $ |(\alpha^{(i)})^{n-n_1}|<1 $, since $ n>n_1$, while $ |2^{m}(\alpha^{(i)})^{1-n_1}|\geq 2$.

Since
$$
h(\gamma_{1})\leq h(f_{k}(\alpha)) +h(\alpha^{n-n_{1}}-1) 
< 3\log k +(n-n_{1})\dfrac{\log\alpha}{k}+\log 2,
$$
it follows that
$$
kh(\gamma_{1}) < 6k\log k + (n - n_1)\log\alpha < 6k\log k + 2.95 \times 10^{11} k^4 \log^{2}k(1+\log n).
$$
So, we can take $ A_{1}:= 3\times 10^{11} k^4 \log^{2}k(1+\log n) $. Further, as before, we take $ A_{2} :=\log 2 $ and $ A_{3}: = k\log2 $. Finally, by recalling that $ m<n $, we can take $ B:=n $.

We then get that
$$\log|\Lambda_{1}|>-1.4\times 30^6 \times 3^{4.5}\times k^{3}(1+\log k)(1+\log n)(3\times 10^{11} k^4 \log^{2}k(1+\log n))(\log 2)^2,$$
which yields
$$ \log |\Lambda_{1}|>-4.13 \times 10^{22} k^7\log^3 k(1+\log n)^2.$$
Comparing this with \eqref{fala6}, we get that
$$(m-m_{1})\log 2 < 4.2\times 10^{22} k^7\log^3 k(1+\log n)^{2}.$$

\textbf{Case 2.} $\min \lbrace (n-n_1) \log \alpha , (m-m_1) \log 2 \rbrace  = (m - m_{1} ) \log 2$.

\medskip

In this case, we write \eqref{fala1} as
\begin{eqnarray*}
\left|f_{k}(\alpha)\alpha^{n-1} - 2^{m} +2^{m_{1}}\right| &=& \left|(f_{k}(\alpha)\alpha^{n-1} -F_{n}^{(k)}) + (F_{n_{1}}^{(k)} - f_{k}(\alpha)\alpha^{n_{1}-1})+ f_{k}(\alpha)\alpha^{n_{1}-1} \right| \\
&<&\dfrac{1}{2}+\dfrac{1}{2}+\alpha^{n_{1}-1} ~~<~~\alpha^{n_{1}},
\end{eqnarray*}
so that
\begin{eqnarray}\label{fala7}
&&\left|f_{k}(\alpha)(2^{m-m_{1}}-1)^{-1}\alpha^{n-1}2^{-m_{1}} - 1\right|<\dfrac{\alpha^{n_{1}}}{2^{m}-2^{m_{1}}}\leq \dfrac{2\alpha^{n_{1}}}{2^{m}}<\alpha^{n_{1}-n+6}.
\end{eqnarray}
The above inequality \eqref{fala7} suggests once again studying a lower bound for the absolute value of
$$
\Lambda_{2} := f_{k}(\alpha)(2^{m-m_{1}}-1)^{-1}\alpha^{n-1}2^{-m_{1}} - 1.
$$
We again apply Matveev's theorem with the following data
$$
t: =3,~ \gamma_{1}: =f_{k}(\alpha)(2^{m-m_{1}}-1)^{-1},~ \gamma_{2}: = \alpha, ~\gamma_{3}: = 2, ~~ b_{1}:=1,~b_{2}:=n-1, b_{3}:=-m_{1}.
$$
We can again take $ B:=n $ and  $ \mathbb{K} := \mathbb{Q}(\alpha) $, so that $ D:=k $. We also note that, if $ \Lambda_{2} =0 $, then $ f_{k}(\alpha) = \alpha^{-(n-n_{1})} 2^{m_{1}} (2^{m-m_{1}}-1) $ implying that $ f_{k}(\alpha) $ is an algebraic integer, which is not the case. Thus, $ \Lambda_{2} \neq 0 $.

Now, we note that
$$
h(\gamma_{1})\leq  h(f_{k}(\alpha))+h(2^{m-m_{1}}-1)
<3\log k +(m-m_{1}+k)\dfrac{\log 2}{k}.
$$
Thus, $ kh(\gamma_{1})< 4k\log k + (m-m_{1})\log 2  < 3 \times 10^{11} k^4\log^2k(1+\log n)$, and so we can take $ A_{1} := 3 \times 10^{11} k^4\log^2k(1+\log n) $. As before, we take $ A_{2}: = \log 2 $ and $ A_{3} := k\log  2 $.
It then follows from Matveev's theorem, after some calculations, that
$$
\log |\Lambda_{2}| > -4.13\times 10^{22}k^7\log^3 k(1+\log n)^2.
$$
From this and \eqref{fala7}, we obtain that
$$(n-n_{1})\log\alpha < 4.2 \times 10^{22}k^7\log^3 k(1+\log n)^2.$$
Thus in both Case $ 1 $ and Case $ 2 $, we have
\begin{eqnarray}\label{fala8}
\min\{(n-n_{1})\log\alpha , (m-m_{1})\log 2\} & < & 4.3\times 10^{11} k^4 \log^{2}k(1+\log n),\\
\max\{(n-n_{1})\log\alpha , (m-m_{1})\log 2\} & < & 4.2\times 10^{22}k^7\log^3 k(1+\log n)^2.\nonumber
\end{eqnarray}
We now finally rewrite equation \eqref{fala1} as
$$
\left|f_{k}(\alpha)\alpha^{n-1} -f_{k}(\alpha)\alpha^{n_{1}-1}-2^{m}+2^{m_{1}}\right| = \left|(f_{k}(\alpha)\alpha^{n-1} - F_{n}^{(k)})+(F_{n_{1}}^{(k)} - f_{k}(\alpha)\alpha^{n_{1}-1})\right| < 1.
$$
We divide through both sides by $ 2^{m}-2^{m_{1}} $ getting
\begin{eqnarray}\label{fala9}
&& \left|\dfrac{f_{k}(\alpha)(\alpha^{n-n_{1}}-1)}{2^{m-m_{1}}-1}\alpha^{n_{1}-1}2^{-m_{1}} - 1\right|<\dfrac{1}{2^{m}-2^{m_{1}}} \leq \dfrac{2}{2^{m}} <2^{5 - 0.8n},
\end{eqnarray}
since $n < 1.2m+4$. To find a lower--bound on the left--hand side of \eqref{fala9} above, we again apply Theorem \ref{Matveev11} with the data
$$
t:=3,~ \gamma_{1}: =\dfrac{f_{k}(\alpha)(\alpha^{n-n_{1}}-1)}{2^{m-m_{1}}-1},~\gamma_{2} := \alpha, ~\gamma_{3} := 2, ~ b_{1}:=1,~b_{2}:=n_{1}-1, b_{3}:=-m_{1}.
$$
We also take $ B:=n $ and we take  $ \mathbb{K} := \mathbb{Q}(\alpha) $ with $ D := k $. From the properties of the logarithmic height function, we have that
\begin{eqnarray*}
kh(\gamma_{1})&\leq& k\left(h(f_{k}(\alpha))+h(\alpha^{n-n_1}-1)+h(2^{m-m_{1}}-1)\right)\\
&<&3k\log k +(n-n_{1})\log\alpha +k(m-m_{1})\log2 + 2k\log2\\
&<&5.3\times 10^{22}k^8\log^3 k(1+\log n)^2,
\end{eqnarray*}
where in the above chain of inequalities we used the bounds \eqref{fala8}. So we can take $ A_{1} :=5.3\times 10^{22}k^8\log^3 k(1+\log n)^2  $, and certainly as before we take $ A_{2} := \log 2 $ and $ A_{3}: = k\log  2 $. We need to show that if we put
$$
\Lambda_{3}:=\dfrac{f_{k}(\alpha)(\alpha^{n-n_{1}}-1)}{2^{m-m_{1}}-1}\alpha^{n_{1}-1}2^{-m_{1}} - 1,
$$
then $ \Lambda_{3} \neq 0 $. To see why $ \Lambda_{3} \neq 0$, note that otherwise, we would get the relation 
$$ 
f_{k}(\alpha)(\alpha^{n-n_{1}}-1) = 2^{m_{1}}\alpha^{1-n_{1}}(2^{m-m_{1}}-1).
$$ 
Again, as for the case of $ \Lambda_{1} $, conjugating the above relation with an automorphism $ \sigma $ of the Galois group of $ \Psi_{k}(x) $ over $ \mathbb{Q} $ such that $ \sigma(\alpha) = \alpha^{(i)} $ for some $ i\geq 2 $, and then taking absolute values, we get that $ |f_{k}(\alpha^{(i)})((\alpha^{(i)})^{n-n_1}-1)| = |2^{m_1}(\alpha^{(i)})^{1-n_1}(2^{m-m_{1}}-1)| $. This cannot hold true because in the left--hand side we have $ |f_{k}(\alpha^{(i)})||(\alpha^{(i)})^{n-n_1}-1|<2 $, while in the right--hand side we have $ |2^{m_{1}}||(\alpha^{(i)})^{1-n_1}||2^{m-m_1}-1|\geq 2 $. Thus,
$ \Lambda_{3} \neq 0 $. Then Theorem \ref{Matveev11} gives
$$\log |\Lambda_{3}|>-1.4\times 30^{6}\times 3^{4.5}k^{11}(1+\log k)(1+\log n)\left(5.3\times 10^{22}\log^3 k(1+\log n)^2\right)(\log 2)^{2},$$
which together with \eqref{fala9} gives
$$(0.8n - 5)\log 2 < 7.3\times 10^{33} k^{11}\log^{4}k(1+\log n)^{3}.$$
The above inequality leads to
\begin{eqnarray*}
n < 5.1\times 10^{34} k^{11}\log^{4}k\log^{3}n,
\end{eqnarray*}
which can be equivalently written as
\begin{eqnarray}\label{fala10}
\dfrac{n}{(\log n)^{3}} & < & 5.1\times 10^{34} k^{11}\log^{4}k.
\end{eqnarray}
If $ A\geq 10^{30} $, the inequality
$$
\dfrac{x}{(\log x)^{3}} < A ~~\text{   yields  } ~~ x < 16A\log^{3}A.
$$
Thus, taking $ A: = 5.1\times 10^{34} k^{11}\log^{4}k $, inequality \eqref{fala10} yields
\begin{eqnarray}
n & < & 2.8\times 10^{41}k^{11}\log^{7}k.
\end{eqnarray}
We then record what we have proved so far as a lemma.
\begin{lemma}\label{lemma12}
If $ (n,m,n_{1},m_{1}, k) $ is a solution in positive integers to equation \eqref{Problem}  with $ (n,m)\neq (n_{1},m_{1}) $, $ n>n_{1}\geq 2 $, $ m>m_{1}\geq 0 $ and $ k\geq 4 $, we then have that $ n < 2.8\times 10^{41}k^{11}\log^{7}k$.
\end{lemma}

\section{Reduction of the bounds on $ n $}

\subsection{The cutoff $k$}

We have from the above that Baker's method gives
$$
n < 2.8\times 10^{41}k^{11}\log^{7}k.
$$
Imposing that the above amount is at most $2^{(k-5)/3}$, which would imply inequality \eqref{hip}, we get
$$
2.8^{3} \times 10^{123} k^{33} (\log k)^{21} < 2^k,
$$
leading to $k>790$.


We now reduce the bounds and to do so we make use of Lemma \ref{Dujjella} several times.

\subsection{The Case of small $ k $}

We next treat the cases when $ k\in[4,790] $. We note that for these values of the parameter $ k $, Lemma \ref{lemma12} gives us absolute upper bounds for $ n $. However, these upper bounds are so large that we wish to reduce them to a range where the solutions can be identified by using a computer. To do this, we return to \eqref{fala3} and put
\begin{eqnarray}
\Gamma &:=&(n-1)\log\alpha - m\log 2 + \log\left(f_{k}(\alpha)\right).
\end{eqnarray}
For technical reasons we assume that $ \min\{n-n_{1}, m-m_{1} \} \geq 20 $. In the case that this condition fails, we consider one of the following inequalities instead:
\begin{itemize}
\item[(i)] if $ n-n_{1} < 20 $ but $ m-m_{1} \geq 20 $, we consider \eqref{fala6};
\item[(ii)]if $ n-n_{1} \geq 20 $ but $ m-m_{1} < 20 $, we consider \eqref{fala7};
\item[(iii)]if $ n-n_{1} < 20 $ and $ m-m_{1} < 20 $, we consider \eqref{fala9}.
\end{itemize}
Let us start by considering \eqref{fala3}. Note that $ \Gamma\neq 0 $; thus we distinguish the following cases. If $ \Gamma > 0 $, then $ e^{\Gamma} - 1>0$, so from \eqref{fala3} we obtain
$$0<\Gamma<e^{\Gamma} -1 <\max\{\alpha^{n_{1}-n+6}, 2^{m_{1}-m+1}\}. $$
Suppose now that $ \Gamma<0 $. Since $ \Lambda = |e^{\Gamma}-1 | <1/2$, we get that $ e^{|\Gamma|}<2 $. Thus,
$$0<|\Gamma|\leq e^{|\Gamma|} -1 =e^{|\Gamma|}|e^{\Gamma}-1| <2\max\{\alpha^{n_{1}-n+6}, 2^{m_{1}-m+1}\}. $$
In any case, we have that the inequality
\begin{eqnarray}\label{arg1}
0<|\Gamma| < 2\max\{\alpha^{n_{1}-n+6}, 2^{m_{1}-m+1}\}
\end{eqnarray}
always holds. Replacing $ \Gamma $ in the above inequality by its formula and dividing through by $ \log 2 $, we conclude that
$$0 < \left|(n-1)\left(\dfrac{\log\alpha}{\log 2}\right) - m + \dfrac{\log (f_{k}(\alpha))}{\log 2}\right|< \max\{200\cdot \alpha^{-(n-n_{1})}, ~8\cdot 2^{-(m-m_{1})}\}.$$
We apply Lemma \ref{Dujjella} with the data
$$k\in[4, 790], \qquad \tau_k: = \dfrac{\log\alpha}{\log 2}, \qquad \mu_k := \dfrac{\log (f_{k}(\alpha))}{\log 2}, \qquad (A_k, B_k) := ( 200, \alpha) \quad \text{or} \quad (8, 2).$$

We also put $ M_{k}: = \lfloor 2.8\times 10^{41}k^{11}\log^{7}k \rfloor $, which is upper bound on $ n $ by Lemma \ref{lemma12}. From the fact that $ \alpha $ is a unit in $ \mathcal{O}_{\mathbb{K}} $, the ring of integers of $ \mathbb{K} $, ensures that $ \tau_k $ is an irrational number. Furthermore, $ \tau_k $ is transcendantal by Gelfond--Schneider Theorem. A computer search in \textit{Mathematica} showed that the maximum value of $ \lfloor \log(200q/\varepsilon)/\log \alpha\rfloor $ is $ < 1571 $ and the maximum value of $ \lfloor \log(8q/\varepsilon)/\log 2\rfloor $ is $ < 1566$. Therefore, either
\begin{eqnarray*}
n-n_{1} < \dfrac{\log(200q/\varepsilon)}{\log\alpha} < 1571, ~~\text {or  } m-m_{1} < \dfrac{\log(8q/\varepsilon)}{\log2}< 1566.
\end{eqnarray*}
Thus, we have that either $ n-n_{1}\leq 1571 $, or $ m-m_{1}\leq 1566 $.

First, let us assume that $ n-n_{1}\leq 1571$. In this case we consider the inequality \eqref{fala6} and assume that $ m-m_{1}\geq 20 $. We put
\begin{eqnarray*}
\Gamma_{1}&=& (n_{1}-1)\log\alpha- m\log2 + \log(f_{k}(\alpha)(\alpha^{n-n_{1}}-1)).
\end{eqnarray*}
By the same arguments used for proving \eqref{arg1}, from \eqref{fala6} we get
\begin{eqnarray*}
0<|\Gamma_{1}|< \dfrac{4}{2^{m-m_{1}}},
\end{eqnarray*}
and so
\begin{eqnarray}\label{arg2}
~~~0<\left|(n_{1}-1)\left(\dfrac{\log\alpha}{\log2}\right) - m + \cfrac{\log(f_{k}(\alpha)(\alpha^{n-n_{1}}-1))}{\log2}\right|<8\cdot 2^{-(m-m_{1})}.
\end{eqnarray}
As before, we keep the same $ \tau_k $, $ M_k $, $ (A_k, B_k) := (8, 2) $ and put
\begin{eqnarray*}
\mu_{k, l}= \dfrac{\log(f_{k}(\alpha)(\alpha^{l}-1))}{\log 2}, \qquad k\in [4, 790] \quad {\rm and} \quad l\in[1, 1566].
\end{eqnarray*}
We now apply Lemma \ref{Dujjella} to inequality \eqref{arg2} for the values of $k\in[4,790]$ and $ l\in [1, 1571]$. A computer search with \textit{Mathematica} revealed that the maximum value of $ \lfloor \log(Aq/\varepsilon)/\log B \rfloor $ over the values of $k\in[4,790]$ and $ l\in [1, 1571] $ is $ < 1570$. Hence, $ m-m_{1} \leq 1570 $.

Now let us assume that $ m-m_{1}\leq 1566 $. In this case, we consider the inequality \eqref{fala7} and assume that $ n-n_{1}\geq 20 $. We put
\begin{eqnarray*}
\Gamma_{2}&=& (n-1)\log\alpha -m_{1}\log2 +\log (f_{k}(\alpha)(2^{m-m_{1}}-1)).
\end{eqnarray*}
Then, by the same arguments as before, we get
\begin{eqnarray*}
0<|\Gamma_{2}|<\dfrac{2\alpha^{6}}{\alpha^{n-n_{1}}}.
\end{eqnarray*}
Replacing $ \Gamma_{2} $ in the above inequality by its formula and dividing through by $ \log 2 $, we finally get that
\begin{eqnarray*}
~~~~~~~~~~0<\left| (n-1)\left(\dfrac{\log\alpha}{\log 2}\right) - m_{1}+ \dfrac{\log(f_{k}(\alpha)(2^{m-m_{1}}-1))}{\log2}\right| < 114\cdot \alpha^{-(n-n_{1})}.
\end{eqnarray*}
We apply Lemma \ref{Dujjella} with the same $ \tau_k $, $ M_k $, $ (A_k, B_k) := (114, \alpha) $ and put
\begin{eqnarray*}
\mu_{k, l}= \dfrac{\log(f_{k}(\alpha)(2^{l}-1))}{\log 2},  \qquad k\in [4, 790] \quad {\rm and} \quad l\in[1, 1566].
\end{eqnarray*}
As before, a computer search with \textit{Mathematica} revealed that the maximum value of
$$
\lfloor \log(Aq/\varepsilon)/\log B \rfloor, \qquad {\rm for } \quad k\in[4,790] \quad {\rm and} \quad l \in [1, 1566]
$$
is $ < 1574 $. Hence, $ n-n_{1}\leq 1574 $.

To conclude the above computations, we first got that either $ n-n_{1}\leq 1571 $ or $ m-m_{1}\leq 1566 $.
If $ n-n_{1}\leq 1571 $, then $ m-m_{1}\leq 1570 $, and if $ m-m_{1}\leq 1566 $, then $ n-n_{1}\leq 1574 $.
Thus, in conclusion, we always have that
\begin{eqnarray*}
n-n_{1}\leq 1574 \qquad {\rm and} \qquad m-m_{1}\leq 1570.
\end{eqnarray*}
Finally, we go to \eqref{fala9} and put
\begin{eqnarray*}
\Gamma_{3}&=& (n_{1}-1)\log\alpha - m_{1}\log2 + \log\left(\dfrac{f_{k}(\alpha)(\alpha^{n-n_{1}}-1)}{2^{m-m_{1}}-1}\right).
\end{eqnarray*}
Since $ n>1600 $, from \eqref{fala9} we conclude that
\begin{eqnarray*}
0<|\Gamma_{3}|< \dfrac{2^{6}}{2^{0.8n}}.
\end{eqnarray*}
Hence,
\begin{eqnarray*}
0<\left|(n_{1}-1)\left(\dfrac{\log\alpha}{\log 2}\right) -m_{1}+\dfrac{\log(f_{k}(\alpha)(\alpha^{l}-1)/(2^{j}-1))}{\log 2}\right|<(2^6/\log 2)\cdot 2^{-n},
\end{eqnarray*}
where $ (l, j): = (n-n_{1},~m-m_{1}) $. We apply Lemma \ref{Dujjella} with the same $ \tau_k $, $ M_k $, $ (A_k, B_k) := (2^6/\log 2, 2)$ and
\begin{eqnarray*}
\mu_{k, l,j} = \dfrac{\log(f_{k}(\alpha)(\alpha^{l}-1)/(2^{j}-1))}{\log 2} \quad \text{for} \quad k\in [4, 790],  \quad l\in[1, 1574]\quad  {\rm and} \quad  j\in [1, 1570].
\end{eqnarray*}
With the help of \textit{Mathematica} we find that the maximum value of
$$
\lfloor \log (114q/\varepsilon)/\log 2 \rfloor, \quad {\rm for} \quad k\in [4, 790],  \quad l\in[1, 1574]\quad  {\rm and} \quad  j\in [1, 1570]
$$
is  $ < 1574$. Thus, $ n< 1574 $, which contradicts the assumption that $ n > 1600 $ in Section 5.

We finish the resolution of the Diophantine equation \eqref{Problem}, for this case, with the following procedure.
Consider the following equivalent equation to \eqref{Problem}
$$
F_n^{(k)} - F_{n_1}^{(k)} = 2^m - 2^{m_1}.
$$
For $k\in [4, 790]$ and $n\in [k+2, 1600]$, let the set
$$
F_{n,k} := \left\{F_n^{(k)} - F_{n_1}^{(k)} ~~({\rm mod} ~ 10^{20}) : n_1 \in [2, n-1]\right\},
$$
and
$$
D_{n, k} := \left\{ 2^m-2^{m_1} ~~({\rm mod} ~ 10^{20}) : m \in [\lfloor c(n-4)\rceil, \lfloor c(n-1)+1 \rceil], ~~m_1\in[0, m-1] \right\}
$$
with $c = \log\alpha / \log2$. Note that we have used \eqref{fala2} to define the range of $m$ in $D_{n, k}$. As in all computations of this paper, with the help of Mathematica, we looked for all $(n, k)$ the intersections
$F_{n, k} \cap D_{n, k}$. After an extensive search, we obtain that $F_{n, k} \cap D_{n, k}$ contains only the solutions corresponding to the families (i)--(iv) in the statement of Theorem \ref{Main} for the current range of the variables.

This completes the proof in the case of small $ k $.

\subsection{The Case of large $ k $.}

In this case we assume that $ k> 790 $, we have already shown that the Diophantine equation \eqref{Problem} has only the solutions listed in Theorem \ref{Main}.

\section*{acknowledgements}
We thank the referee for pointing out some errors in a previous version of this manuscript. M.~D. was supported by the FWF grant F5510-N26, which is part of the special research program (SFB), ``Quasi Monte Carlo Metods: Theory and Applications''. M.~D. would also like to thank his supervisor Prof. Dr. Robert Tichy for the encouragement and the useful comments and remarks that greatly improved in the quality of this paper.
C.~A.~G. was supported in part by Project 71079 (Universidad del Valle). F.~L. was supported by grant CPRR160325161141 and an A-rated scientist award both from the NRF of South Africa and by grant no. 17-02804S of the Czech Granting Agency.


\begin{thebibliography}{}
%
%



\bibitem{BD69} A. Baker and H. Davenport, {\it The equations $3x^2-2 =y^2$ and $8x^2-7=z^2$}, \newblock{\em Quart. J. Math. Oxford Ser}. {\bf 20}, No. 2,  pp. 129--137 (1969).

\bibitem{bawu07}
A.~Baker and G.~W{\"u}stholz,
\newblock {\em Logarithmic forms and {D}iophantine geometry}, volume~9 of {\em
  New Mathematical Monographs}.
\newblock Cambridge University Press, Cambridge (2007).

\bibitem{BBL17} J. J. Bravo and F. Luca, {\it Powers of two in generalized Fibonacci sequences}, Revista Colombiana de Matem\'aticas  {\bf 46}, No.1,  pp. 67--79 (2012).

\bibitem{BravoLuca13} J. J. Bravo and F. Luca, {\it Coincidences in generalized Fibonacci sequences}, J. Number Theory {\bf 133}, No. 6,  pp. 2121--2137 (2013).

\bibitem{Mihailescu}
P.~Mih\u ailescu {\it Primary cyclotomic units and a proof of Catalans conjecture} \newblock{ Journal f\"ur die reine und angewandte Mathematik (Crelles Journal)}, 2004.{\bf572} (2006): 167-195. Retrieved 24 Jul. 2017, from doi:10.1515/crll.2004.048

\bibitem{Luca16} J.~J.~Bravo, F.~Luca, and K.~Yaz{\'a}n. {\it On a problem of Pillai with Tribonacci numbers and powers of 2}. 
\newblock {\em Bull. Korean Math. Soc.} \textbf{54}, No. 3, pp. 1069--1080 (2017). doi:10.4134/BKMS.b160486.

\bibitem{cpz16} K.~C.~Chim, I.~Pink and V.~Ziegler. {\it On a variant of Pillai's problem}. \newblock {\em  Int. J. Number Theory}, \textbf{13}, No.7, pp. 1711--1727 (2017), doi:10.1142/S1793042117500981.

\bibitem{Howard:2011} C. Cooper and F. T. Howard, {\it Some identities for $r-$Fibonacci numbers}. \newblock{\em Fibonacci Quart.} \textbf{49}, No.3 , pp. 231--243 (2011).

\bibitem{Luca15} M.~Ddamulira, F.~Luca and M.~Rakotomalala. {\it On a problem of Pillai with Fibonacci numbers and powers of 2}. \newblock {\em Proc. Math. Sci.}, \textbf{127}, No.3, pp. 411--421 (2017), doi:10.1007/s12044-017-0338-3.

\bibitem{Dresden2014} G. P. Dresden and Zhaohui Du, {\it A simplified Binet formula for $k-$generalized Fibonacci numbers}, J. Integer Sequences {\bf 17} (2014), Article 14.4.7.

\bibitem{dujella98} A. Dujella and A. Peth\H o, {\it A generalization of a theorem of Baker and Davenport}, Quart. J. Math. Oxford Ser. \textbf{49} , No.3, pp. 291--306 (1998).

\bibitem{Gomez} C. A. G\'omez Ru\'{\i}z and F. Luca, {\it On the largest prime factor of the ratio of two generalized Fibonacci numbers}, \newblock{\em J. Number Theory} {\bf 152} , pp. 182--203 (2015).

\bibitem{Herschfeld:1935} A. Herschfeld. {\it The equation $2^x - 3^y = d$}. \newblock {\em Bull. Amer. Math. Soc. }, {\bf 41}, pp. 631 (1935).

\bibitem{Herschfeld:1936} A. Herschfeld. {\it The equation $2^x - 3^y = d$}. \newblock {\em Bull. Amer. Math. Soc. }, {\bf 42}, pp. 231--234 (1936).

\bibitem{MatveevII} E. M. Matveev, {\it An explicit lower bound for a homogeneous rational linear form in the logarithms of algebraic numbers}, II, Izv. Ross. Akad. Nauk Ser. Mat. \textbf{64}, No.6 , pp. 125--180 (2000); translation in Izv. Math. \textbf{64} ,  No. 6, pp. 1217--1269 (2000).

\bibitem{Pillai:1936} S.~S. Pillai. {\it On $a^x - b^y = c$}. \newblock {\em J. Indian Math. Soc. (N.S.)}, {\bf2}: pp. 119--122 (1936).

\bibitem{Pillai:1937} S.~S. Pillai. {\it A correction to the paper On $a^x - b^y = c$}. \newblock {\em J. Indian Math. Soc. (N.S.)}, {\bf 2}, pp. 215, (1937).

\end{thebibliography}

\def\cprime{$'$}


 \end{document}